\documentclass[11pt, reqno]{amsart}
\usepackage{indentfirst, amssymb, amsmath, amsthm, mathrsfs, setspace, indentfirst, enumerate,  mathrsfs, amsmath, amsthm}
\usepackage[bookmarksnumbered, colorlinks, plainpages]{hyperref}
\usepackage{mathrsfs}

\newtheorem*{theoA}{Theorem A}
\newtheorem*{theoB}{Theorem B}
\newtheorem*{theoC}{Theorem C}
\newtheorem*{theoD}{Theorem D}
\newtheorem*{theoE}{Theorem E}

\newtheorem*{cor A}{Corollary A}
\newtheorem*{cor B}{Corollary B}

\newtheorem{theo}{Theorem}[section]
\newtheorem{lem}{Lemma}[section]

\newtheorem{rem}{Remark}[section]

\newcommand{\ol}{\overline}

\newcommand{\be}{\begin{equation}}
\newcommand{\ee}{\end{equation}}
\newcommand{\beas}{\begin{eqnarray*}}
\newcommand{\eeas}{\end{eqnarray*}}
\newcommand{\bea}{\begin{eqnarray}}
\newcommand{\eea}{\end{eqnarray}}

\numberwithin{equation}{section}
\begin{document}

\title[S\MakeLowercase{tudy on the certain type of nonlinear algebraic partial differential equation}......]{\LARGE S\Large\MakeLowercase{tudy on the certain type of nonlinear algebraic partial differential equation in} $\mathbb{C}^m$}
\date{}
\author[S. M\MakeLowercase{ajumder}, D. P\MakeLowercase{ramanik} \MakeLowercase{and} N. S\MakeLowercase{arkar}]{S\MakeLowercase{ujoy} M\MakeLowercase{ajumder}$^*$, D\MakeLowercase{ebabrata} P\MakeLowercase{ramanik} \MakeLowercase{and} N\MakeLowercase{abadwip} S\MakeLowercase{arkar}}
\address{Department of Mathematics, Raiganj University, Raiganj, West Bengal-733134, India.}
\email{sm05math@gmail.com, sjm@raiganjuniversity.ac.in}
\address{Department of Mathematics, Raiganj University, Raiganj, West Bengal-733134, India.}
\email{debumath07@gmail.com}
\address{Department of Mathematics, Raiganj University, Raiganj, West Bengal-733134, India.}
\email{naba.iitbmath@gmail.com}

\renewcommand{\thefootnote}{}
\footnote{*\emph{Corresponding Author}: Sujoy Majumder.}
\renewcommand{\thefootnote}{\arabic{footnote}}
\setcounter{footnote}{0}

\begin{abstract} In the paper, using Nevanlinna's value distribution theory of meromorphic functions in $\mathbb{C}^m$, we study for the existence of entire solutions $f$ in $\mathbb{C}^m$ of the following algebraic partial differential equation
\[f^n(z)+P_d(f(z))=p(z)e^{\langle \alpha,\ol z\rangle},\]
where $P_d(f)$ is an algebraic differential polynomial in $f$ of degree $d \leq n-2$, $n \geq 3$ is an integer, $p$ is a non-zero polynomial, $\alpha=(\alpha_{1},\ldots,\alpha_{m})\neq (0,\ldots,0)$ and $\langle \alpha,\ol z\rangle=\sideset{}{_{k=1}^{m}}{\sum}\alpha_{1k} z_k$. Also in the paper, we study for the non-existence of entire solutions $f$ in $\mathbb{C}^m$ of the following algebraic partial differential equation
\[f^n(z)+P_d(f(z))=p_1(z)e^{\langle \alpha, \ol z\rangle}+p_2(z)e^{\langle \beta, \ol z\rangle},\]
where $P_d(f)$ is an algebraic differential polynomial of degree $d \leq n-3$, $n \geq 4$ is an integer, $p_1$ and $p_2$ are two non-zero polynomials, $\alpha=(\alpha_{11},\ldots,\alpha_{1m})\neq (0,\ldots,0)$ and $\beta=(\alpha_{21},\ldots,\alpha_{2m})\neq (0,\ldots,0)$ such that $\alpha_{1i}\neq 0$, $\alpha_{2i}\neq 0$ and $\alpha_{1i}/\alpha_{2i}\not\in\mathbb{Q}$ for all $i\in\mathbb{Z}[1,m]$. Our findings extend and improve the results of Li and Yang (J. Math. Anal. Appl., 320 (2006) 827-835) and Zhang and Liao (Taiwanese J. Math., 15 (5) (2011), 2145-2157) into higher dimensions.
\end{abstract}

\thanks{Typeset by \AmS -\LaTeX}
\maketitle

\section{{\bf Introduction}}
We define $\mathbb{Z}_+=\mathbb{Z}[0,+\infty)=\{n\in \mathbb{Z}: 0\leq n<+\infty\}$ and $\mathbb{Z}^+=\mathbb{Z}(0,+\infty)=\{n\in \mathbb{Z}: 0<n<+\infty\}$.
On $\mathbb{C}^m$, we define
\[\partial_{z_i}=\frac{\partial}{\partial z_i},\ldots, \partial_{z_i}^{l_i}=\frac{\partial^{l_i}}{\partial z_i^{l_i}}\;\;\text{and}\;\;\partial^{I}=\frac{\partial^{|I|}}{\partial z_1^{i_1}\cdots \partial z_m^{i_m}}\]
where $l_i\in \mathbb{Z}^+\;(i=1,2,\ldots,m)$ and $I=(i_1,\ldots,i_m)\in\mathbb{Z}^m_+$ be a multi-index such that $|I|=\sum_{j=1}^m i_j$.

\medskip
Let $K$ be a field. We know that the ring $K[y,y_1,\ldots,y_n]$ is a unique factorization domain and its elements are called polynomials in $n+1$ variables. Every element in this ring can be written as a sum $\Psi=\sum_{i\in I}c_i\Psi_i$
where $I$ is a finite set of distinct elements in $\mathbb{Z}_{+}^{n+1}$, $c_i\in K$ and $\Psi_i(y,y_1,\ldots,y_n)=y^{i_0}y_1^{i_1}\ldots y_n^{i_n}$, $i=(i_0,i_1,\ldots,i_n)\in \mathbb{Z}_+^{n+1}$.
The elements $c_i(i\in I)$ are called the coefficients. We call each term $c_i\Psi_i$ a monomial and if $c_i\neq 0$ we define the degree of this monomial to be $\deg(\Psi_i)=|i|=i_0+i_1+\ldots+i_n$ (see \cite{HLY}).

Take a positive integer $m$ and given distinct multi-indices $\nu= \{ \nu _1, \ldots ,\nu_n \}\subset  \mathbb{Z}_+^m$ with 
$0<|\nu_1| \leq |\nu_2| \leq \ldots \leq |\nu_n|$. Let $f:\mathbb{C}^m\to\mathbb{P}^1$ be a non-constant meromorphic function. We take $K=\mathbb{M}(\mathbb{C}^m)$, field of meromorphic functions in $\mathbb{C}^m$ and choose the indeterminant $y$, $y_1$,$\ldots$,$y_n$ as $y=f$, $y_k= \partial^{\nu_k}(f)\; (k=1,\ldots,n)$.
As usual, a polynomial $\Psi \big(f, \partial^{\nu_1}(f),\ldots,\partial^{\nu_n}(f) \big) \in \mathbb{M}(\mathbb{C}^m)[f, \partial^{\nu_1}(f),\ldots,\partial^{\nu_n}(f)] \subset \mathbb{M} (\mathbb{C}^{m})$
can be expressed as follows
\[P_d(f)=\Psi\big(f, \partial^{\nu_1}(f),\ldots,\partial^{\nu_n}(f) \big)=\sideset{}{_{i\in I}}{\sum}c_i\Psi_i(f),\]
where $\Psi_i(f)=f^{i_0}(\partial^{\nu_1}(f))^{i_1}\ldots (\partial^{\nu_n}(f))^{i_n}$, $c_i\in S(f)$ and $d=\max\{\deg(\Psi_i): i\in I\}$. 
The polynomial $P_d(f)$ is called a \text{differential} \text{polynomial} of $f$ and is said to be \text{differential monomial} of $f$ if $P_d(f)$ is monomial. Also the polynomial $P_d(f)$, is called an \text{algebraic differential polynomial} of $f$ if $c_i$'s are polynomials for $i\in I$. Here $d$ is called the degree of the polynomial $P_d(f)$.

\medskip
First of all it is quiet difficult to ensure the existence of solutions of a given differential equation and it is also interesting to find out the solutions if the solutions exist. Now recall a special type of non-linear differential equation 
\bea\label{e1.0} f^n+P_d(f)=h,\eea
where $h\in\mathbb{M}(\mathbb{C})$ and $P_d(z,f)$ is a differential polynomial in $f$ in $\mathbb{C}$ of degree $d$ has become a matter of increasing interest among the researchers. Since 1970's, Nevanlinna's value distribution theory have been utilized by many researchers (see, e.g., \cite{HKL,LY1,CCY1,CCY2,cc7}) to study solvability and existence of entire or meromorphic solutions of differential equation (\ref{e1.0}), where $g$ is a given entire or meromorphic function in $\mathbb{C}$.

\smallskip
In \cite{HKL}, it was proved that $f_{1}(z)=-\frac{\sqrt{3}}{2}\cos z-\frac{1}{2}\sin z$ is also a solution of $4f^{3}(z)+3f^{(2)}(z)=-\sin 3z$. In 2004, Yang and Li \cite{cc7} proved that this equation admits exactly three entire solutions namely $f_{1}(z)$, $f_{2}(z)=\sin z$ and $f_{3}(z)=\frac{\sqrt{3}}{2}\cos z-\frac{1}{2}\sin z$. 
Since $-\sin 3z$ is a linear combination of $\exp(3\iota z)$ and $\exp(-3\iota z)$, so it is interesting to find out all entire solutions of the following general equation 
\bea\label{e1.1}f^{n}+P_d(f)=p_1e^{\alpha_1}+p_2e^{\alpha_2},\eea
where $P_d(f)$ is a differential polynomial with small functions of $f$ as the coefficients such that $d\leq n-1$, $p_1$, $p_2$ are rational functions in $\mathbb{C}$ and $\alpha_1$, $\alpha_2$ are polynomials in $\mathbb{C}$.

\smallskip
It is interesting and quite difficult problem to ensure the solvability and existence for entire and meromorphic solutions of Eq. (\ref{e1.1}) and find out the solutions if the solutions exist. In the last twenty years, many researchers (see \cite{CG1,ps2,LYZ1,LYe1,LLW1,cc7}) have shown great interest in the special type of non-linear differential equation (\ref{e1.1}).

\smallskip
In 2024, Yang and Li \cite{cc7} obtained the following result.
\begin{theoA}\cite[Theorem 5]{cc7} Let $n\geq 3$ be an integer, $P_{d}(f)$ be an algebraic differential polynomial of degree $d\leq n-3$, $b$ be a meromorphic function in $\mathbb{C}$, and $\lambda, c_1, c_2$ be three non-zero constants. Then the following Eq.
\bea\label{2004}f^n(z)+P_{n-3}(f(z))=b(z)\left(c_1e^{\lambda z}+c_2e^{-\lambda z}\right),\eea
has no transcendental entire solutions $f$ in $\mathbb{C}$, that satisfy $T(r, b)=S(r, f)$.
\end{theoA}

In 2006, Li and Yang \cite{ps2} derived similar conclusions, when the term on the right-hand side of (\ref{2004}) is replaced by $p_1 \exp(\alpha_1 z)+p_2 \exp(\alpha_2 z)$, where $p_1, p_2$ are non-zero polynomials in $\mathbb{C}$, $\alpha_1, \alpha_2$ are two constants with $\alpha_1 / \alpha_2 \neq$ rational number. We now recall the result.

\begin{theoB}\cite[Theorem 1]{ps2} Let $n\geq 4$ be an integer and $P_d(f)$ denote an algebraic differential polynomial of degree $d\leq n-3$. Let $p_1$ and $p_2$ be two non-zero polynomials in $\mathbb{C}$, $\alpha_1$ and $\alpha_2$ be two non-zero constants with $\alpha_1/\alpha_2\neq$ rational. Then the following Eq.
\bea\label{ddd1} f^n(z) +P_d(f(z))= p_1\exp(\alpha_1z) + p_2\exp(\alpha_2z)\eea
has no transcendental entire solutions in $\mathbb{C}$.
\end{theoB}

In the same paper, Li and Yang \cite{ps2} also studied the case when $\alpha_1/\alpha_2$ is a rational number and obtained the following result.
\begin{theoC}\cite[Theorem 2]{ps2} Let $n\geq 3$ be an integer and $P_d(f)$ denote an algebraic differential polynomial of degree $d\leq n-2$. Let $p_1$ and $p_2$ be non-zero polynomials in $\mathbb{C}$, $\alpha_1$ and $\alpha_2$ be non-zero constants with $\alpha_1/\alpha_2= s/t$, where $s,t\in\mathbb{N}$ such that $s/t\geq n/(n - 1)$ and $s/t \neq n/d$. Then the Eq. (\ref{ddd1}) has no transcendental entire solutions in $\mathbb{C}$.
\end{theoC}

Theorems B-C gave some conditions to judge the non-existence of entire solutions of (\ref{ddd1}). In Theorems C, though the restriction on $n$ has been reduced to $n\geq 3$, it is required that $\alpha_1/\alpha_2$ is a positive rational number. Also in the same paper, Li and Yang \cite{ps2} obtained entire solutions to the new equations under conditions $n\geq 3$ and $d\leq n-2$. We recall the result as follows:

\begin{theoD}\cite{ps2} Let $a,p_1,p_2, \lambda\in\mathbb{C}\backslash \{0\}$. Then the equation
\bea\label{ddd2} f^3(z)+af^{(2)}(z)= p_1e^{\lambda z}+p_2e^{-\lambda z}\eea
has transcendental entire solutions in $\mathbb{C}$ if and only if the condition $p_1p_2 + (a\lambda^2/27)^3 = 0$
holds. Moreover, if this condition holds, then the solutions of the Eq. (\ref{ddd2}) are 
\[f(z)=\rho_je^{\lambda z / 3}-\left(a \lambda^2/27 \rho_j\right)e^{-\lambda z / 3}, \quad j=1,2,3,\]
where $\rho_j\;(j=1,2,3)$ are the cubic roots of $p_1$.
\end{theoD}

In 2011, Zhang and Liao \cite{ZL} proved that Theorem D does not hold if $f^{(2)}$ is replaced by $f^{(1)}$ in (\ref{ddd2}) and the result is as follows:
\begin{theoE}\cite[Theorem 2]{ZL} Let $a,p_1,p_2, \lambda\in\mathbb{C}\backslash \{0\}$. Then the following Eq.
\beas f^3(z)+af^{(1)}(z)= p_1e^{\lambda z}+p_2e^{-\lambda z}\eeas
does not have any transcendental entire solution in $\mathbb{C}$.
\end{theoE}

\smallskip
In recent years, the Nevanlinna value distribution theory in $\mathbb{C}^m$ has emerged as a prominent and rapidly growing area of research in complex analysis. This field has garnered significant attention due to its deep theoretical insights and wide-ranging applications in mathematics and related disciplines. Researchers have been particularly intrigued by its potential to extend classical results from one complex variable to higher-dimensional settings, as a result, this topic has become a focal point for contemporary studies in $\mathbb{C}^m$. Recently partial differential equations, partial difference equations and partial differential-difference equations are the current interests among researchers by utilizing Nevanlinna value distribution theory in several complex variables, see \cite{CLL}, \cite{HZ1}, \cite{FL}.

\medskip
We firstly recall some basis notions in several complex variables (see \cite{HLY,MR,WS}).
On $\mathbb{C}^m$, the exterior derivative $d$ splits $d= \partial+ \bar{\partial}$ and twists to $d^c= \frac{\iota}{4\pi}\left(\bar{\partial}- \partial\right)$. Clearly $dd^{c}= \frac{\iota}{2\pi}\partial\bar{\partial}$. A non-negative function $\tau: \mathbb{C}^m\to \mathbb{R}[0,b)\;(0<b\leq \infty)$ of class $\mathbb{C}^{\infty}$ is said to be an exhaustion of $\mathbb{C}^m$ if $\tau^{-1}(K)$ is compact whenever $K$ is. 
An exhaustion $\tau_m$ of $\mathbb{C}^m$ is defined by $\tau_m(z)=||z||^2$. The standard Kaehler metric on $\mathbb{C}^m$ is given by $\upsilon_m=dd^c\tau_m>0$. On $\mathbb{C}^m\backslash \{0\}$, we define $\omega_m=dd^c\log \tau_m\geq 0$ and $\sigma_m=d^c\log \tau_m \wedge \omega_m^{m-1}$. For any $S\subseteq \mathbb{C}^m$, let $S[r]$, $S(r)$ and $S\langle r\rangle$ be the intersection of $S$ with respectively the closed ball, the open ball, the sphere of radius $r>0$ centered at $0\in \mathbb{C}^m$.

\smallskip
Let $f$ be a holomorphic function on $G(\not=\varnothing)$, where $G$ is an open subset of $\mathbb{C}^m$. Then we can write $f(z)=\sum_{i=0}^{\infty}P_i(z-a)$, where the term $P_i(z-a)$ is either identically zero or a homogeneous polynomial of degree $i$. Certainly the zero multiplicity $\mu^0_f(a)$ of $f$ at a point $a\in G$ is defined by $\mu^0_f(a)=\min\{i:P_i(z-a)\not\equiv 0\}$.

\medskip
Let $f$ be a meromorphic function on $G$. Then there exist holomorphic functions $g$ and $h$ such that $hf=g$ on $G$ and $\dim_z h^{-1}(\{0\})\cap g^{-1}(\{0\})\leq m-2$. Therefore the $c$-multiplicity of $f$ is just $\mu^c_f=\mu^0_{g-ch}$ if $c\in\mathbb{C}$ and $\mu^c_f=\mu^0_h$ if $c=\infty$. The function $\mu^c_f: \mathbb{C}^m\to \mathbb{Z}$ is nonnegative and is called the $c$-divisor of $f$. If $f\not\equiv 0$ on each component of $G$, then $\nu=\mu_f=\mu^0_f-\mu^{\infty}_f$ is called the divisor of $f$. We define 
$\text{supp}\; \nu=\text{supp}\;\mu_f=\ol{\{z\in G: \nu(z)\neq 0\}}$.

For $t>0$, the counting function $n_{\nu}$ is defined by
\beas n_{\nu}(t)=t^{-2(m-1)}\int_{A[t]}\nu \upsilon_m^{m-1},\eeas
where $A=\text{supp}\;\nu$. The valence function of $\nu$ is defined by 
\[N_{\nu}(r)=N_{\nu}(r,r_0)=\int_{r_0}^r n_{\nu}(t)\frac{dt}{t}\;\;(r\geq r_0).\]

For $a\in\mathbb{P}^1$, we write $n_{\mu_f^a}(t)=n(t,a;f)$, if $a\in\mathbb{C}$ and $n_{\mu_f^a}(t)=n(t,f)$, if $a=\infty$. Also we write $N_{\mu_f^a}(r)=N(r,a;f)$ if $a\in\mathbb{C}$ and $N_{\mu_f^a}(r)=N(r,f)$ if $a=\infty$.
For $k\in\mathbb{N}$, define the truncated multiplicity functions on $\mathbb{C}^m$ by $\mu_{f,k}^a(z)=\min\{\mu_f^a(z),k\}$,
and write the truncated counting functions $n_{\nu}(t)=n_k(t,a;f)$, if $\nu=\mu_{f,k}^a$ and $n_{\nu}(t)=\ol{n}(t,a;f)$, if $\nu=\mu_{f,1}^a$. Also we write
 $N_{\nu}(t)=N_k(t,a;f)$, if $\nu=\mu_{f,k}^a$ and $N_{\nu}(t)=\ol{N}(t,a;f)$, if $\nu=\mu_{f,1}^a$.

\medskip
An algebraic subset $X$ of $\mathbb{C}^m$ is defined as a subset
\[X=\left\lbrace z\in\mathbb{C}^m: P_j(z)=0,\;1\leq j\leq l\right\rbrace\]
with finitely many polynomials $P_1(z),\ldots, P_l(z)$.
A divisor $\nu$ on $\mathbb{C}^m$ is said to be algebraic if $\nu$ is the zero divisor of a polynomial. In this case the counting
function $n_{\nu}$ is bounded (see \cite{GK1,ST1}).

\medskip
With the help of the positive logarithm function, we define the proximity function of $f$ by
\[m(r, f)=\mathbb{C}^m\langle r; \log^+ | f | \rangle=\int_{\mathbb{C}^m\langle r\rangle} \log^+ |f|\;\sigma_m.\]

The characteristic function of $f$ is defined by $T(r, f)=m(r,f)+N(r,f)$. 
We define $m(r,a;f)=m(r,f)$ if $a=\infty$ and $m(r,a;f)=m(r,1/(f-a))$ if $a$ is finite complex number. Now if $a\in\mathbb{C}$, then the first main theorem of Nevanlinna theory states that $m(r,a;f)+N(r,a;f)=T(r,f)+O(1)$, where $O(1)$ denotes a bounded function when $r$ is sufficiently large. Clearly if $f$ is non-constant then $T(r, f) \rightarrow \infty$ as $r \rightarrow$ $\infty$. Obviously 
$\lim\limits _{r \rightarrow \infty} \frac{T(r, f)}{\log r}<\infty$ if $f$ is a rational function and $\lim\limits_{r \rightarrow \infty} \frac{T(r, f)}{\log r}=\infty$, if $f$ is transcendental. We define the order of $f$ by
\[\rho(f):=\limsup _{r \rightarrow \infty} \frac{\log T(r, f)}{\log r}.\]

Let $S(f)=\{g:\mathbb{C}^m\to\mathbb{P}^1\;\text{meromorphic}:\parallel T(r,g)=o(T(r,f))\}$, where $\parallel$ indicates that the equality holds only outside a set of finite measure on $\mathbb{R}^+$.

\section{\bf Main results}
We define 
\[\partial_{z_i}(F(z))=\frac{\partial F(z)}{\partial z_i},\ldots, \partial_{z_i}^{l_i}(F(z))=\frac{\partial^{l_i} F(z)}{\partial z_i^{l_i}},\]
where $l_i\in \mathbb{Z}^+$, $i\in\mathbb{Z}[1,m]$. Let $S_{m}=\{(u_1,u_2,\ldots,u_m)\in \mathbb{C}^{m}: (u_1, u_2,\cdots,u_m)\neq (0,0,\ldots,0)\}$. For $u\in S_m$, we define
\[\partial_u(f) =\sideset{}{_{j=1}^m}{\sum}u_{j}\partial_{z_j}(f).\]

The $k$-th order derivative $\partial^k_u(f)$ of $f$ is defined by $\partial^k_u(f)=\partial_u(\partial_u^{k-1}(f))$
inductively. Note that if $\dim(\mathbb{C}^m)=1$, then $\partial^k_u(f)=u_1f^{(k)}$.\par

\medskip
One may ask whether there exist corresponding results for Theorems A-E in the case of higher dimension? In the paper, we give an affirmative answer of this question and obtain some results which generalize and extend Theorems A-E from $\mathbb{C}$ to $\mathbb{C}^m$.

\medskip
In the paper, we consider the following partial differential equation 
\bea\label{dd2} f^n(z)+P_d(f(z))=p_1(z)e^{\langle \alpha, \ol z\rangle}+p_2(z)e^{\langle \beta,\ol z\rangle},\eea
where $p_1$, $p_2$ are non-zero polynomials in $\mathbb{C}^m$, $\alpha=(\alpha_{11},\ldots,\alpha_{1m})\in S_m$ and $\beta=(\alpha_{21},\ldots,\alpha_{2m})\in S_m$. Clearly the Eq. (\ref{dd2}) is an extension of the Eqs. (\ref{2004})-(\ref{ddd2}) into the higher dimensions.

\medskip
Our first objective is to find out the possible solution of Eq. (\ref{dd2}), when the right side contains only one term. Now we state our first result regarding this.

\begin{theo}\label{t5} Let $P_d(f)$ denote an algebraic differential polynomial of degree $d\leq n-2$, where $n\geq 3$ is an integer, $p$ be non-zero polynomial in $\mathbb{C}^m$ and let $\alpha=(\alpha_{1},\ldots,\alpha_{m})\in S_m$. Then the partial differential equation 
\bea\label{edd2} f^n(z)+P_d(f(z))=p(z)e^{\langle \alpha, \ol z\rangle}\eea
has transcendental entire solutions and the solutions are $f(z)=q(z)e^{\langle \alpha, \ol z\rangle/n}$, where $q(z)$ is a polynomial in $\mathbb{C}^m$ such that $q^n(z)=p(z)$. In this case $P_d(f)\equiv 0$.
\end{theo}


\medskip
Our next objective is to generalize and extend Theorems A-E into higher dimensions. In this regard, we now state our next results.

\begin{theo}\label{t1}
Let $P_d(f)$ denote an algebraic differential polynomial of degree $d\leq n-3$, where $n\geq 4$ is an integer and let $p_1$ and $p_2$ be two non-zero polynomials in $\mathbb{C}^m$. Suppose $\alpha=(\alpha_{11},\ldots,\alpha_{1m})$ and $\beta=(\alpha_{21},\ldots,\alpha_{2m})$ such that $\alpha_{1i}\neq 0$, $\alpha_{2i}\neq 0$ and $\alpha_{1i}/\alpha_{2i}\not\in\mathbb{Q}$ for all $i\in\mathbb{Z}[1,m]$. Then the Eq. (\ref{dd2}) has no transcendental entire solutions in $\mathbb{C}^m$.
\end{theo}

\begin{theo}\label{t2} 
Let $P_d(f)$ denote an algebraic differential polynomial of degree $d\leq n-2$, where $n\geq 3$ is an integer and let $p_1$ and $p_2$ be two non-zero polynomials in $\mathbb{C}^m$. Suppose $\alpha=(\alpha_{11},\ldots,\alpha_{1m})$ and $\beta=(\alpha_{21},\ldots,\alpha_{2m})$ such that $\alpha_{1i}\neq 0$, $\alpha_{2i}\neq 0$ for all $i\in\mathbb{Z}[1,m]$ and $t\alpha_1=s\alpha_2$, where $s, t\in\mathbb{N}$ with $s / t \geq n/(n-1)$ and $s / t \neq n/d$. Then the Eq. (\ref{dd2}) has no transcendental entire solutions.
\end{theo}

\begin{theo}\label{t3} Let $u\in S_m$, $p_1, p_2$, $\lambda_i\;(i=1,2,\cdots,m)$ be non-zero constants and $l\geq 2$ be an even integer and let $\lambda=(\lambda_1,\ldots,\lambda_m)$. Then the partial differential equation
\bea\label{dt41}f^3(z)+\partial_u^l(f(z))=p_1e^{\langle \lambda, \ol z\rangle}+p_2e^{-\langle \lambda, \ol z\rangle}\eea
has transcendental entire solutions and the solutions are
$f(z)=ce^{\langle \lambda, \ol z\rangle/3}+de^{-\langle \lambda, \ol z\rangle/3}$
where $c$ and $d$ are non-zero constants such that $c^3=p_1$, $d^3=p_2$ and $3cd+m^l=0$.
\end{theo}

\begin{theo}\label{t4} Let $u\in S_m$, $p_1, p_2$, $\lambda_i\;(i=1,2,\cdots,m)$ be non-zero constants and $l\geq 1$ be a odd integer and let $\lambda=(\lambda_1,\ldots,\lambda_m)$. Then the partial differential equation
\bea\label{de} f^3(z)+\partial_u^l(f(z))=p_1e^{\langle \lambda, \ol z\rangle}+p_2e^{-\langle \lambda, \ol z\rangle}\eea
has no transcendental entire solutions in $\mathbb{C}^m$.
\end{theo}

\begin{rem} If we take $\dim(\mathbb{C}^m)=1$, then Theorems \ref{t3} and \ref{t4} improve Theorems D and E respectively in a large scale.
\end{rem}

\section{\bf{Auxiliary lemmas}}
For the proofs of our Theorems, we need the following Lemmas.
\begin{lem}\label{L.1} \cite[Lemma 1.37]{HLY} Let $f:\mathbb{C}^m\to\mathbb{P}^1$ be a non-constant meromorphic function and $I=(\alpha_1,\ldots,\alpha_m)\in \mathbb{Z}^m_+$ be a multi-index. Then for any $\varepsilon>0$, we have
\[m\left(r,\partial^I(f)/f\right)\leq |I|\log^+T(r,f)+|I|(1+\varepsilon)\log^+\log T(r,f)+O(1)\]
for all large $r$ outside a set $E$ with $\int_E d\log r<\infty$.
\end{lem}

\begin{lem}\label{L.2} \cite[Lemma 1.2]{HY1} Let $f:\mathbb{C}^m\to\mathbb{P}^1$ be a non-constant meromorphic function and let $a_1,a_2,\ldots,a_q$ be different points in $\mathbb{P}^1$. Then
\beas \parallel (q-2)T(r,f)\leq \sideset{}{_{j=1}^{q}}{\sum} \ol N(r,a_j;f)+O(\log (rT(r,f))).\eeas
\end{lem}

\begin{lem}\label{L.3} \cite[Theorem 1.26]{HLY} Let $f:\mathbb{C}^m\to\mathbb{P}^1$ be non-constant meromorphic function. Assume that 
$R(z, w)=\frac{A(z, w)}{B(z, w)}$. Then
\beas T\left(r, R_f\right)=\max \{p, q\} T(r, f)+O\Big(\sideset{}{_{j=0}^p}{\sum} T(r, a_j)+\sideset{}{_{j=0}^q}{\sum}T(r, b_j)\Big),\eeas
where $R_f(z)=R(z, f(z))$ and two coprime polynomials $A(z, w)$ and $B(z,w)$ are given
respectively as follows: $A(z,w)=\sum_{j=0}^p a_j(z)w^j$ and $B(z,w)=\sum_{j=0}^q b_j(z)w^j$.
\end{lem}

Let $f$ be a non-constant meromorphic function in $\mathbb{C}^m$. Define complex differential polynomials of $f$ 
as follows:
\bea\label{cl1} P(f)=\sideset{}{_{\mathbf{p} \in I}}{\sum} a_{\mathbf{p}} f^{p_0}\big(\partial^{\mathbf{i}_1}(f)\big)^{p_1} \cdots\big(\partial^{\mathbf{i}_l}(f)\big)^{p_l}, \quad \mathbf{p}=\left(p_0, \ldots, p_l\right) \in \mathbb{Z}_{+}^{l+1},\eea
\bea\label{cl2} Q(f)=\sideset{}{_{\mathbf{q} \in J}}{\sum} c_{\mathbf{q}} f^{q_0}\big(\partial^{\mathbf{j}_1}(f)\big)^{q_1} \cdots\big(\partial^{\mathbf{j}_s}(f)\big)^{q_s}, \quad \mathbf{q}=\left(q_0, \ldots, q_s\right) \in \mathbb{Z}_{+}^{s+1}\eea
and 
\bea\label{c13} B(f)=\sideset{}{_{k=0}^{n}}{\sum} b_k f^k,\eea
where $I, J$ are finite sets of distinct elements and $a_{\mathbf{p}}, c_{\mathbf{q}}, b_k$ are rational functions in $\mathbb{C}^m$ such that $b_n \not \equiv 0$.

In 2014, Hu and Yang \cite[Lemma 2.1]{ps1} generalised Clunie-lemma to high dimension. Note that if $\rho(f)<+\infty$, then by Lemma \ref{L.1}, we have $m(r,\partial^I(f)/f)=O(\log r)$ as $r\to +\infty$. Now by using Lemma \ref{L.1}, the following lemma can be easily derived from the proof of Lemma 2.1 \cite{ps1} (see also \cite{HY0,HLY,BQL}).

\begin{lem}\label{L.4} Let $f:\mathbb{C}^m\to\mathbb{P}^1$ be a non-constant meromorphic function such that $\rho(f)<+\infty$. Assume that $f$ satisfies the differential equation $B(f)Q(f)=P(f)$,
where $P(f)$, $Q(f)$ and $B(f)$ are defined as in (\ref{cl1}), (\ref{cl2}) and (\ref{c13}) respectively.
If $\deg(P(f)) \leq n=\deg(B(f))$, then $m(r, Q(f))=O(\log r)$ as $r\to +\infty$.
\end{lem}

\begin{lem}\label{L.5}\cite[Theorem 1.2]{ps1} Suppose that $f$ is meromorphic and not constant in $\mathbb{C}^m$, that $g=f^n+P_{n-1}(f)$, where $P_{n-1}(f)$ is a differential polynomial of degree at most $n-1$ and that 
\[\parallel N(r,f)+ N(r,0;g)=o(T(r,f)).\]

Then
$g=\left(f+a/n\right)^n$,
where $a$ is a meromorphic function in $\mathbb{C}^m$ determined by the terms of degree $n-1$ in $P_{n-1}(f)$ and by $g$ and $a\in S(f)$.
\end{lem}

\begin{lem}\label{L.6}
Let $f,g:\mathbb{C}^m\to\mathbb{P}^1$ be non-constant meromorphic functions and let $m_1$ and $m_2$ be two positive integers such that $m_1\geq 3$ and $m_2\geq 3$. Suppose $a, b:\mathbb{C}^m\to\mathbb{P}^1$ be non-zero meromorphic functions such that $a,b\in S(f)$.
Then the functional equation 
\bea\label{sdd1}af^{m_1}+bg^{m_2}=1\eea
holds only when $m_1=m_2=3$. If $f$ and $g$ are entire functions, then Eq. (\ref{sdd1}) cannot hold.
\end{lem}

\begin{proof} Using Lemma \ref{L.3} to (\ref{sdd1}), we get
\bea\label{sdd.1} \parallel\;m_1T(r,f)+o(T(r,f))=m_2T(r,g)+o(T(r,g)).\eea

Let
\bea\label{sdd.2} h=(af^{m_1}-1)/af^{m_1}.\eea

Then $h$ is a non-constant meromorphic function in $\mathbb{C}^m$. Using Lemma \ref{L.3} to (\ref{sdd.2}), we get 
\[\parallel T(r,h)+o(T(r,h))=m_1T(r,f)+o(T(r,f)).\]

Now from (\ref{sdd1}) and (\ref{sdd.2}), it is easy to deduce that 
\[\parallel\;\ol N(r,h)\leq \ol N\left(r,0, af^{m_1}\right)\leq N\left(r,0,a\right)+\ol N\left(r,0,f^{m_1}\right)=\ol N\left(r,0,f\right)+o(T(r,f)),\]
\[\parallel\;\ol N\left(r,0,h\right)=\ol N\left(r,1,af^{m_1}\right)\leq \ol N\left(r,0,g^{m_2}\right)+N\left(r,0,b\right)\leq \ol N\left(r,0,g\right)+o(T(r,g))\]
and 
\[\parallel\;\ol N\left(r,1,h\right)\leq \ol N(r,f^{m_1})+\ol N(r,a)=\ol N(r,f)+o(T(r,f)). \] 

Therefore by Lemma \ref{L.2}, we get
\bea\label{sdd.3} \parallel\;m_1T(r,f)&=&T(r,h)+o(T(r,h))\\&\leq& \ol N(r,h)+\ol N(r,0,h)+\ol N(r,1,h)+o(T(r,h))\nonumber\\&\leq&
\ol N(r,0,f)+\ol N(r,0,g)+\ol N(r,f)+o(T(r,f))+o(T(r,g)).\nonumber
\eea

Now in view of first main theorem and using (\ref{sdd.1}) to (\ref{sdd.3}), we get
\beas \parallel\;\left(m_1-2-m_1/m_2\right)T(r,f)\leq o(T(r,f))\eeas
from which, we get $m_1-2-m_1/m_2\leq 0$, i.e., 
\bea\label{sdd.4}\frac{2}{m_1}+\frac{1}{m_2}\geq 1.\eea  

Since $m_1\geq 3$ and $m_2\geq 3$, (\ref{sdd.4}) yields $m_1=m_2=3$. 
If $f$ is entire, then from (\ref{sdd.3}), we get
\[ \parallel\;m_1T(r,f)\leq \ol N(r,0,f)+\ol N(r,0,g)+o(T(r,f))+o(T(r,g))\]
and so from (\ref{sdd.1}), we get 
\[\parallel\;\left(m_1-1-m_1/m_2\right)T(r,f)\leq o(T(r,f)).\]

Consequently $\frac{1}{m_1}+\frac{1}{m_2}\geq 1$, which is impossible for $m_1\geq 3$ and $m_2\geq 3$.
Hence the proof.
\end{proof}

\begin{lem}\label{L.8} Let $n\geq 3$ be an integer and let $P_{d}(f)$ be a differential polynomial with rational functions in $\mathbb{C}^m$ as its coefficients such that $d\leq n-2$. If $f$ is an entire solution of the Eq. (\ref{dd2}), then $f$ is of finite order.
\end{lem}

\begin{proof} It is easy to decide that $f$ is a transcendental. If possible, suppose $\rho(f)=+\infty$. Then for any given large $M_{0}>0$ and sufficiently large $r$, we have 
\[T(r,f)>r^{M_{0}}.\]
 
Note that $T(r,e^{\langle \alpha, \ol z\rangle})=m(r,e^{\langle \alpha, \ol z\rangle})=O(r)$. Similarly $T(r,e^{\langle \beta, \ol z\rangle})=O(r)$  If we take $g=p_1e^{\langle \alpha, \ol z\rangle}+p_2e^{\langle \beta, \ol z\rangle}$, then by Lemma \ref{L.3}, we get 
\[T(r,g)\leq T(r,e^{\langle \alpha, \ol z\rangle})+T(r,e^{\langle \beta, \ol z\rangle})+O(\log r)\]
and so $T(r,g)=O(r)$. Let us take $M_0>d$. Then $\frac{T(r,g)}{T(r,f)}\rightarrow 0$ as $r\rightarrow \infty$. This shows that $g\in S(f)$. Now from (\ref{dd2}), we have 
\bea\label{o1} f^{n}=\mathcal{Q}(f),\eea
where $\mathcal{Q}(f)=g-P_d(f)$. Now using Lemma 2.1 \cite{ps1} to (\ref{o1}), we get $m(r,f)=o(T(r,f))$ and so $T(r,f)=o(T(r,f))$, which is absurd. So $\rho(f)<+\infty$. Hence the proof.
\end{proof}

\begin{lem}\label{L.7} \cite[Corollary 4.5]{BCL} Let $a_1,a_2,\ldots,a_n$ be $n$ meromorphic functions
in $\mathbb{C}^m$ and $g_1,g_2,\ldots,g_n$ be $n$ entire functions in $\mathbb{C}^m$ satisfying
\[\sideset{}{_{i=1}^{n}}{\sum}a_ie^{g_i}\equiv 0.\]
If for all $1\leq i\leq n$,
\[T(r,a_i)=o(T(r,e^{g_j-g_k})), \;\; j\neq k,\]
then $a_i\equiv 0$ for $1\leq i\leq n$.
\end{lem}

\section {{\bf Proof of Theorem \ref{t5}}}
Let $f$ be a transcendental entire solution of (\ref{edd2}). Applying the operator $\partial_{z_i}$ on the both sides of (\ref{edd2}), we have
\bea\label{edd3} nf^{n-1}\partial_{z_i}(f)+\partial_{z_i}(P_d(f))=(p\alpha_{i}+\partial_{z_i}(p))e^{\langle \alpha, \ol z\rangle},\eea
for any $i\in\mathbb{Z}[1,m]$. Eliminating $e^{\langle \alpha, \ol z\rangle}$ from (\ref{edd2}) and (\ref{edd3}), we have
\bea\label{edd4} \left(\left(p\alpha_{i}+\partial_{z_i}(p)\right)f-np\partial_{z_i}(f)\right)f^{n-1}=-\left(p\alpha_{i}+\partial_{z_i}(p)\right)P_d(f)+p\partial_{z_i}(P_d(f)),\eea
for any $i\in\mathbb{Z}[1,m]$.

\smallskip
First we suppose $(p\alpha_{i}+\partial_{z_i}(p))f-np\partial_{z_i}(f)\not\equiv 0$ for some $i\in\mathbb{Z}[1,m]$. Now using Lemma 2.1 \cite{ps1} to (\ref{edd4}), we get
\bea\label{edd5} \left\{\begin{array}{clcr} m\left(r,(p\alpha_{i}+\partial_{z_i}(p))f-np\partial_{z_i}(f)\right)=o(T(r,f)),\\
m\left(r,(p\alpha_{i}+\partial_{z_i}(p))f^2-npf\partial_{z_i}(f)\right)=o(T(r,f)).\end{array}\right.\eea

Note that $f$ is an entire function. Therefore using Lemma \ref{L.3} and (\ref{edd5}), we get
\beas T(r,f)\leq T(r,(p\alpha_{i}+\partial_{z_i}(p))f^2-npf\partial_{z_i}(f))
+T(r,(p\alpha_{i}+\partial_{z_i}(p))f-np\partial_{z_i}(f))=o(T(r,f)),
\eeas
which is impossible.

\smallskip
Next we suppose $(p\alpha_{i}+\partial_{z_i}(p))f-np\partial_{z_i}(f)\equiv 0$ for any $i\in\mathbb{Z}[1,m]$. Let $F=f^n$ and $G=pe^{\langle \alpha, \ol z\rangle}$. If possible, suppose $F$ and $G$ are linearly independent. Then by Corollary 1.40 \cite{HLY}, there exists $j\in \mathbb{Z}[1, m]$ such that
\[W(F,G)=\left|\begin{array}{ll} F&G\\\partial_{z_j}(F)&\partial_{z_j}(G)\end{array}\right|\not\equiv 0\]
and so $e^{\langle \alpha, \ol z\rangle}\left((p\alpha_{j}+\partial_{z_j}(p)\big)f-np\partial_{z_j}(f)\right)\not\equiv 0$,
i.e., $(p\alpha_{j}+\partial_{z_j}(p))f-np\partial_{z_j}(f)\not\equiv 0$,
which is absurd. Hence $F$ and $G$ are linearly dependent and so $f^n=cpe^{\langle \alpha, \ol z\rangle}$, where $c\in\mathbb{C}\backslash \{0\}$.
Now from (\ref{edd2}), we get
\[(1-1/c)f^n=-P_d(f).\]

If $c\neq 1$, then using Lemma 2.1 \cite{ps1}, we get $m(r,f)=o(T(r,f))$ and so $T(r,f)=o(T(r,f))$, which is impossible. Hence $c=1$ and so $P_d(f)\equiv 0$. Finally we have $f(z)=q(z)e^{\langle \alpha, \ol z\rangle/n}$, where $q(z)$ is a polynomial in $\mathbb{C}^m$ such that $q^n(z)=p(z)$. Hence the proof.

\section {{\bf Proof of Theorem \ref{t1}}}
Let $f$ be a transcendental entire solution of (\ref{dd2}). Clearly by Lemma \ref{L.8}, we have $\rho(f)<+\infty$. Applying the operator $\partial_{z_i}$ on the both sides of (\ref{dd2}), we get
\bea\label{dd3} nf^{n-1}\partial_{z_i}(f)+\partial_{z_i}(P_d(f))=(p_1\alpha_{1i}+\partial_{z_i}(p_1))e^{\langle \alpha,\ol z\rangle}
+(p_1\alpha_{2i}+\partial_{z_i}(p_2))e^{\langle \beta,\ol z\rangle}\eea
for any $i\in\mathbb{Z}[1,m]$. Eliminating $e^{\langle \alpha,\ol z\rangle}$ and $e^{\langle \beta,\ol z\rangle}$ from (\ref{dd2}) and (\ref{dd3}), we get respectively
\bea\label{dd4} \big(p_1\alpha_{1i}+\partial_{z_i}(p_1)\big)f^n-np_1f^{n-1}\partial_{z_i}(f)+Q_{i,d}(f)=\beta_{1i} e^{\langle \beta,\ol z\rangle}\eea
and
\bea\label{dd5} \big(p_2\alpha_{2i}+\partial_{z_i}(p_2)\big)f^n-np_2f^{n-1}\partial_{z_i}(f)+R_{i,d}(f)=-\beta_{1i}e^{\langle \alpha,\ol z\rangle},\eea
where
\bea\label{dd6}\beta_{1i}=\left(p_1 \alpha_{1i}+\partial_{z_i}(p_1)\right) p_2-\left(p_2 \alpha_{2i}+\partial_{z_i}(p_2)\right) p_1, \eea
\bea\label{dd7} Q_{i,d}(f)=\left(p_1 \alpha_{1i}+\partial_{z_i}(p_1)\right) P_d(f)-p_1\partial_{z_i}(P_d(f))\eea
and
\bea\label{dd8}R_{i,d}(f)=\left(p_2 \alpha_{2i}+\partial_{z_i}(p_2)\right) P_d(f)-p_2\partial_{z_i}(P_d(f)).\eea

By the given conditions, we see that $p_1$ and $p_2$ are non-zero polynomials and $\alpha_{1i}\neq 0$, $\alpha_{2i}\neq 0$ and $\alpha_{1i}/\alpha_{2i}\not\in\mathbb{Q}$ for all $i\in\mathbb{Z}[1,m]$. 
If possible, suppose 
\[(p_1 \alpha_{1i}+\partial_{z_i}(p_1)) p_2-(p_2 \alpha_{2i}+\partial_{z_i}(p_2)) p_1\equiv 0\]
for some $i\in\mathbb{Z}[1,m]$. Then
$\partial_{z_i}(p_1)/p_1-\partial_{z_i}(p_2)/p_2\equiv \alpha_{2i}-\alpha_{1i}\neq 0$. This gives
\bea\label{xx} p_1(z)=p_2(z)e^{(\alpha_{1i}-\alpha_{2i}) z_i+g_i(z)},\eea
where $g_i(z)=g(z_1,\ldots, z_{i-1},z_{i+1},\ldots,z_m)$ is an entire function.
Clearly $(\alpha_{1i}-\alpha_{2i}) z_i+g_i(z)$ is non-constant entire function and so $p_2(z)e^{(\alpha_{1i}-\alpha_{2i}) z_i+g_i(z)}$ is a transcendental entire function. Therefore from (\ref{xx}), we get a contradiction. Hence 
\[(p_1 \alpha_{1i}+\partial_{z_i}(p_1)) p_2-(p_2 \alpha_{2i}+\partial_{z_i}(p_2)) p_1\not\equiv 0\]
and so $\beta_{1i}\not\equiv 0$ for any $i\in\mathbb{Z}[1,m]$.

\smallskip
Again applying the operator $\partial_{z_j}$ on the both sides of (\ref{dd4}), we have
\bea\label{dd9}&& \partial_{z_j}\big(p_1 \alpha_{1i}+\partial_{z_i}(p_1)\big)f^n+
n\big((p_1\alpha_{1i}+\partial_{z_i}(p_1))\partial_{z_j}(f)-\partial_{z_j}(p_1)\partial_{z_i}(f)
-p_1\partial^2_{z_j z_i}(f)\big)f^{n-1}\nonumber\\&&
-n(n-1) p_1 f^{n-2}\partial_{z_j}(f)\partial_{z_i}(f)
+\partial_{z_j}(Q_{i,d}(f))=(\partial_{z_j}(\beta_{1i})+\beta_{1i}\alpha_{2j})e^{\langle \beta, \ol z\rangle},\eea
for any $j\in\mathbb{Z}[1,m]$. Eliminating $e^{\langle \beta, \ol z\rangle}$ from (\ref{dd4}) and (\ref{dd9}), we get
\bea\label{dd10} f^{n-2}\left\lbrace \gamma_{ji} f^2-n\xi_{ji} f+n(n-1) p_1 \beta_{1i}\partial_{z_j}(f)\partial_{z_i}(f)+n p_1\beta_{1i} f \partial^2_{z_j z_i}(f)\right\rbrace=T_{j,i,d}(f),\eea
where $\partial_{z_j z_i}^2(f)=\partial_{z_j}(\partial_{z_i}(f))$,
\[\gamma_{ji}=(\partial_{z_j}(\beta_{1i})+\beta_{1i}\alpha_{2j})(\partial_{z_i}(p_1)+p_1\alpha_{1i})-\beta_{1i}\partial_{z_j}(\partial_{z_i}(p_1)+p_1 \alpha_{1i}\big),\]
\[\xi_{ji}=p_1(\partial_{z_j}(\beta_{1i})+\alpha_{2j} \beta_{1i})\partial_{z_i}(f)+\beta_{1i}(p_1\alpha_{1i}+\partial_{z_i}(p_1))\partial_{z_j}(f)-\partial_{z_j}(p_1)\partial_{z_i}(f))\]
and
\bea\label{dd11} T_{j,i,d}(f)=\beta_{1i}\partial_{z_j}(Q_{i,d}(f))-\big(\partial_{z_j}(\beta_{1i})+\beta_{1i} \alpha_{2j}\big) Q_{i,d}(f).\eea

Set
\bea\label{dd12}\varphi_{ji}= \gamma_{ji} f^2-n\xi_{ji} f+n(n-1) p_1 \beta_{1i}\partial_{z_j}(f)\partial_{z_i}(f)+n p_1\beta_{1i} f \partial^2_{z_j z_i}(f).\eea
 
If possible, suppose $\varphi_{ji}\not\equiv 0$. Now in view of (\ref{dd12}) and using Lemma \ref{L.4} to (\ref{dd10}), we get
$\parallel m(r, \varphi_{ji})=O(\log r)$. Therefore
\bea\label{dd13} \parallel\;T(r,\varphi_{ji})=O(\log r).\eea

Again from (\ref{dd10}), we have $f^{n-3}(\varphi_{ji} f)=T_{j,i,d}(f)$ and so by Lemma \ref{L.4}, we get $\parallel m(r, f\varphi_{ji})=O(\log r)$. Since $f$ is entire, from (\ref{dd13}), we have $\parallel T(r, f\varphi_{ji})=O(\log r)$. Therefore
\[\parallel\;T(r,f)\leq T(r, \varphi_{ji})+T(r, f\varphi_{ji})+O(1)=O(\log r),\]
which is absurd. Hence $\varphi_{ji}\equiv 0$ and so from (\ref{dd10}), we get $T_{j,i,d}(f)\equiv 0$. Therefore (\ref{dd11}) gives
\bea\label{dd14}\beta_{1i}\partial_{z_j}(Q_{i,d}(f))-\big(\partial_{z_j}(\beta_{1i})+\beta_{1i} \alpha_{2j}\big) Q_{i,d}(f)\equiv 0\eea
for any $i,j\in\mathbb{Z}[1,m]$.
Now we consider following two cases.\par

\smallskip
{\bf Case 1.} Let $Q_{i,d}(f)\not\equiv 0$ for all $i\in\mathbb{Z}[1,m]$. Clearly from (\ref{dd7}), we have $P_d(f)\not\equiv 0$. Let 
$F_{1i}=Q_{i,d}(f)$ and $G_{1i}=\beta_{1i}e^{\langle \beta, \ol z\rangle}$.
If possible, suppose $F_{1i}$ and $G_{1i}$ are linearly independent. Then by Corollary 1.40 \cite{HLY}, there exists $j\in \mathbb{Z}[1, m]$ such that $W(F_{1i},G_{1i})\not\equiv 0$.
This implies that 
\[\beta_{1i}\partial_{z_j}(Q_{i,d}(f))-(\partial_{z_j}(\beta_{1i})+\beta_{1i} \alpha_{2j})Q_{i,d}(f)\not\equiv 0,\]
which contradicts (\ref{dd14}). Hence $F_{1i}$ and $G_{1i}$ are linearly dependent. Then there exist a non-zero constants $c_{1i}$ such that $F_{1i}=c_{1i}G_{1i}$, i.e., $Q_{i,d}(f)=c_{1i}\beta_{1i}e^{\langle \beta, \ol z\rangle}$ and so (\ref{dd4}) gives
\bea\label{dd16} f^{n-1}\left\{\left(p_1 \alpha_{1i}+\partial_{z_i}(p_1)\right)f-np_1\partial_{z_i}(f)\right\}=-(1-1/c_{1i}) Q_{i,d}(f).\eea

If $(p_1 \alpha_{1i}+\partial_{z_i}(p_1))f-np_1 \partial_{z_i}(f)\not\equiv 0$ for some $i\in\mathbb{Z}[1,m]$, then using Lemma \ref{L.4} to (\ref{dd16}), we get 
\[\parallel\;m\big(r,\big(p_1 \alpha_{1i}+\partial_{z_i}(p_1)\big) f-n p_1 \partial_{z_i}(f)\big)=O(\log r).\]

Now proceeding with the same arguments as done above, we get a contradiction. Hence
\bea\label{dd17}\big(p_1 \alpha_{1i}+\partial_{z_i}(p_1)\big) f-n p_1\partial_{z_i}(f)\equiv 0\eea
for any $i\in\mathbb{Z}[1,m]$. Let $F_1=f^n$ and $G_1=p_1e^{\langle \alpha,\ol z\rangle}$.
If $F_1$ and $G_1$ are linearly independent, then by Corollary 1.40 \cite{HLY}, there exists $j\in \mathbb{Z}[1, m]$ such that $W(F_1,G_1)\not\equiv 0$ and so
\beas (p_1\alpha_{1j}+\partial_{z_j}(p_1))f-np_1\partial_{z_j}(f)\not\equiv 0,\eeas
which contradicts (\ref{dd17}). Hence $F_1$ and $G_1$ are linearly dependent. Then there exists $c_2\in\mathbb{C}\backslash \{0\}$ such that $F_1=c_2G_1$ and so
\bea\label{dd18} f^n(z)=c_2 p_1(z)e^{\langle \alpha,\ol z\rangle},\;\;\text{i.e.,}\;\;f(z)=cq_1(z)e^{\langle \alpha, \ol z\rangle/n},\eea
where $c^n=c_2$ and $q_1^n=p_1$. Since $P_d(f)$ is an algebraic differential polynomial of degree $d$, using (\ref{dd18}), we may assume that
\bea\label{dd18a} P_d(f(z))=q_{1d}(z)e^{d\langle \alpha, \ol z\rangle/n}+\ldots+q_{11}(z)e^{\langle \alpha, \ol z\rangle/n}+q_{10}(z),\eea
where $q_{1d},\ldots, q_{11}, q_{10}$ are polynomials in $\mathbb{C}^m$. Since $P_d(f)\not\equiv 0$, for the sake of simplicity we may assume that $q_{1d}\not\equiv 0$.
Now using (\ref{dd18}) and (\ref{dd18a}) to (\ref{dd2}), we get
\bea\label{dd18b} (c_2-1) p_1(z)g^n(z)+q_{1d}(z)g^d(z)+\ldots+q_{11}(z)g(z)+q_{10}(z)=p_2(z)e^{\langle \beta, \ol z\rangle},\eea
where $g(z)=e^{\langle \alpha,\ol z\rangle/n}$.
We now divide following two sub-cases.\par

\smallskip
{\bf Sub-case 1.1.} Let $c_2\neq 1$. We define
\beas \tilde F(z)=(c_2-1) p_1(z)g^n(z)+q_{1d}(z)g^d(z)+\ldots+q_{11}(z)g(z).\eeas

If possible suppose that $q_{10}\not\equiv 0$. Then from (\ref{dd18b}), we have 
\[N\left(r,-1,\tilde F/q_{10}\right)=O(\log r).\]

Therefore using first main theorem, Lemmas \ref{L.2} and \ref{L.3}, we get
\beas nT(r,g)&=&T\left(r,\tilde F/q_{10}\right)\\&\leq& \ol N\left(r,\tilde F/q_{10}\right)+\ol N\left(r,0,\tilde F/q_{10}\right)+\ol N\left(r,-1,\tilde F/q_{10}\right)+o(T(r,g))\\&\leq &
\ol N\left(r,0,\tilde F/q_{10}\right)+o(T(r,g))\\&\leq&
\ol N\left(r,0,(c_2-1) p_1(z)g^{n-1}(z)+q_{1d}(z)g^{d-1}(z)+\ldots+q_{11}(z)\right)+o(T(r,g))\\&\leq& (n-1)T(r,g)+o(T(r,g)),
 \eeas
which is impossible. Hence $q_{10}\equiv 0$. Therefore from (\ref{dd18b}), we have 
\beas g(z)\left((c_2-1) p_1(z)g^{n-1}(z)+q_{1d}(z)g^{d-1}(z)+\ldots+q_{11}(z)\right)=p_2(z)e^{\langle \beta,\ol z\rangle}.\eeas

Now proceeding step by step as done above, one can easily prove that $q_{11}\equiv 0$, $\ldots$, $q_{1d}\equiv 0$. 
Since $q_{1d}\not\equiv 0$, we ultimately get a contradiction.

\smallskip
{\bf Sub-case 1.2.} Let $c_2=1$. Then from (\ref{dd18b}), we have 
\bea\label{dd18c} q_{1d}(z)g^d(z)+\ldots+q_{11}(z)g(z)+q_{10}(z)=p_2(z)e^{\langle \beta,\ol z\rangle}.\eea

In this case also, proceeding step by step as done above, one can easily prove that $q_{10}\equiv 0$, $q_{11}\equiv 0$, $\ldots$, $q_{1d-1}\equiv 0$. Thereby from (\ref{dd18c}), we have 
\beas \exp\left(\sideset{}{_{k=1}^{m}}{\sum} \left(d\alpha_{1k}/n-\alpha_{2k}\right) z_k\right)=p_2(z)/q_{1d}(z),\eeas
from which, we get $d\alpha_{1k}=n\alpha_{2k}$, $k\in\mathbb{Z}[1,m]$ and $p_2/q_{1d}\in\mathbb{C}\backslash \{0\}$. So we get a contradiction.\par

\smallskip
{\bf Case 2.} Let $Q_{i,d}(f)\equiv 0$ for some $i\in\mathbb{Z}[1,m]$. For the sake of simplicity, we may assume that $Q_{k,d}(f)\equiv 0$. Then from (\ref{dd7}), we have
\bea\label{dd21}\big(p_1\alpha_{1k}+\partial_{z_k}(p_1)\big)P_d(f)-p_1\partial_{z_k}(P_d(f))\equiv 0.\eea

Now we divide following sub-cases.\par

\smallskip
{\bf Sub-case 2.1.} Let $P_d(f)\equiv 0$. Then from (\ref{dd2}), we have 
\bea\label{dd18d} a(z)\left(f(z)\exp(-\langle \beta,\ol z\rangle/n\right)^{m_1}+b(z)\left(\exp(\langle \alpha-\beta,\ol z\rangle/m_2)\right)^{m_2}=1,\eea
where $a=1/p_2$, $b=-p_1/p_2$, $m_1=n\geq 4$ and $m_2=3$. 
By the given condition, $\langle \alpha-\beta,\ol z\rangle$ is a non-constant polynomial and so $f_2(z)=\exp(\langle \alpha-\beta,\ol z\rangle/m_2)$ is transcendental. If 
\[f_1(z)=f^n(z)e^{-\langle \beta,\ol z\rangle}=p_3(z),\]
where $p_3$ is a non-zero polynomial in $\mathbb{C}^m$, then (\ref{dd18d}) gives $p_1f_2^{m_2}=p_3^{m_1}-p_2$.
Since $p_1$, $p_2$ and $p_3$ are non-zero polynomials and $f_2$ is a transcendental entire function, we get a contradiction. So $f_1$ is a transcendental entire function. Now from (\ref{dd18d}), we get
\bea\label{exx2} af_1^{m_1}+bf_2^{m_2}=1,\eea
where $m_1=n\geq 4$ and $m_2=3$. Now using Lemma \ref{L.6} to (\ref{exx2}), we get a contradiction.\par

\smallskip
{\bf Sub-case 2.2.} Let $P_d(f)\not\equiv 0$. Then from (\ref{dd4}), we have 
\bea\label{exx3} \big(p_1\alpha_{1k}+\partial_{z_k}(p_1)\big)f^n-np_1f^{n-1}\partial_{z_k}(f)=\beta_{1k} e^{\langle \beta,\ol z\rangle}.\eea

Now from (\ref{exx3}), it is easy to prove that the set of zeros of $f$ is algebraic. Therefore from (\ref{exx3}), we may assume that $f(z)=p(z)e^{\langle \beta,\ol z\rangle/n}$, where $p(z)$ is a non-zero polynomial in $\mathbb{C}^m$. Substituting $f(z)=p(z)e^{\langle \beta,\ol z\rangle/n}$ into (\ref{dd2}), we get
\bea\label{exx4} (p^n-p_2)e^{\langle \beta,\ol z\rangle}+\sideset{}{_{s=0}^{d}}{\sum}a_{s}e^{s\langle \beta,\ol z\rangle/n}=p_1e^{\langle \alpha,\ol z\rangle},\eea
where $a_{s}\;(s=0,1,\ldots, d)$ are polynomials functions in $\mathbb{C}^m$. Since $P_d(f)\not\equiv 0$, for the sake of simplicity we may assume that $a_d\not\equiv 0$. Now we consider following sub-cases.

\smallskip
{\bf Sub-case 2.2.1.} Let $p^n\not\equiv p_2$. Now proceeding step by step as done in the proof of Sub-case 1.1, one can easily prove that $a_0\equiv 0$, $a_1\equiv 0,\ldots, a_d\equiv 0$. Since $a_d\not\equiv 0$, we get a contradiction.

\smallskip
{\bf Sub-case 2.2.2.} Let $p^n\equiv p_2$. Then from (\ref{exx4}), we get
\bea\label{exx5} \sideset{}{_{s=0}^{d}}{\sum}a_{s}e^{s\langle \beta,\ol z\rangle/n}=p_1e^{\langle \alpha,\ol z\rangle},\eea

In this case also, proceeding step by step as done in the proof of Sub-case 1.2, one can easily prove that $a_{0}\equiv 0$, $a_{1}\equiv 0$, $\ldots$, $a_{d-1}\equiv 0$. Thereby from (\ref{exx5}), we have 
\beas \exp\left(\sideset{}{_{i=1}^{m}}{\sum} \left(d\alpha_{2i}/n-\alpha_{1i}\right) z_i\right)=p_1(z)/a_{d}(z),\eeas
from which, we get $d\alpha_{2i}=n\alpha_{1i}$, $i\in\mathbb{Z}[1,m]$ and $p_1/a_{d}\in\mathbb{C}\backslash \{0\}$. So we get a contradiction.\par
Hence the proof.

\section {{\bf Proof of Theorem \ref{t2}}}
Let $f$ be a transcendental entire solution of (\ref{dd2}). Clearly by Lemma \ref{L.8}, we have $\rho(f)<+\infty$. Also (\ref{dd3})-(\ref{dd13}) still hold.
We now consider following two cases.\par

\smallskip
{\bf Case 1.} Let $\varphi_{ji}\not\equiv 0$ for some $i,j\in\mathbb{Z}[1,m]$. Now dividing (\ref{dd12}) by $f^2$ and then using Lemma \ref{L.1}, we get $\parallel m\left(r,1/f^2\right)=O(\log r)$, i.e.,
\bea\label{ds27}\parallel m\left(r,1/f\right)=O(\log r).\eea

Since $\alpha_{1j}/\alpha_{2j}=s/t, j=1,2,\ldots,m$, it follows from (\ref{dd4}) and (\ref{dd5}) that
\bea\label{ds28}
&& f^{s(n-1)}\left(\psi_{ii}+Q_{i,d}(f)/f^{n-1}\right)^s \nonumber\\
&& =(-1)^t \beta_{1i}^{s-t}\left(\left(\partial_{z_i}(p_2)+p_2\alpha_{2i}\right)f^n-n p_2 f^{n-1}\partial_{z_i}(f)+R_{i,d}(f)\right)^t,\eea
where 
\bea\label{ds28a} \psi_{1i}=\left(\partial_{z_i}(p_1)+p_1 \alpha_{1i}\right) f-n p_1\partial_{z_i}(f).\eea

Now we consider following sub-cases.

\smallskip
{\bf Sub-case 1.1.} Let $\psi_{1i}\equiv 0$. Then from (\ref{ds28a}), we get $f^n(z)=p_1(z)e^{\alpha_{1i} z_i+h_{i}(z)}$, where $h_i(z)=h(z_1,\ldots, z_{i-1},z_{i+1},\ldots,z_m)$ is a polynomial and so we may assume that $f(z)=p(z)e^{(\alpha_{1i} z_i+h_{i}(z))/n}$, where $p(z)$ is a non-zero polynomial in $\mathbb{C}^m$. Clearly $\alpha_{1i} z_i+h_{i}(z)$ is a non-constant polynomial. On the other hand from (\ref{dd4}), we get 
\bea\label{ds28b} Q_{i,d}(f(z))=\beta_{1i}(z)e^{\langle \beta,\ol z\rangle}\not\equiv 0.\eea

Substituting $f(z)=p(z)e^{(\alpha_{1i} z_i+h_{i}(z))/n}$ into (\ref{ds28b}), we get
\bea\label{ds28c} \sideset{}{_{s=0}^{d}}{\sum}a_{si}(z)e^{s(\alpha_{1i} z_i+h_{i}(z))/n}=\beta_{1i}(z)e^{\langle \beta,\ol z\rangle},\eea
where $a_{si}(z)\;(s=0,1,\ldots, d)$ are polynomials functions in $\mathbb{C}^m$. Since $Q_{i,d}(f)\not\equiv 0$, we may assume that $a_{di}\not\equiv 0$. Now proceeding step by step as done in the proof of Sub-case 1.2 of Theorem \ref{t1}, we can prove that $a_{0i}\equiv 0$, $a_{1i}\equiv 0$, $\ldots$, $a_{d-1i}\equiv 0$. Then (\ref{ds28c}) gives 
\beas a_{di}e^{d(\alpha_{1i} z_i+h_{i}(z))/n}=\beta_{1i}(z)e^{\langle \beta,\ol z\rangle},\eeas
from which, we deduce that $\deg(h_{i})=1$, $d\alpha_{1i}=n\alpha_{2i}$ and $\beta_{1i}/a_{di}\in\mathbb{C}\backslash \{0\}$. Since $d\alpha_{1i}=n\alpha_{2i}$, we get a contradiction.\par

\smallskip
{\bf Sub-case 1.2.} Let $\psi_{1i}\not\equiv 0$. Now using Lemma \ref{L.4} to (\ref{ds28}), we get
\bea\label{ds29} \parallel m\left(r, \psi_{1i}+Q_{i,d}(f)/f^{n-1}\right)=O(\log r).\eea

Since $Q_{i,d}(f)$ is an algebraic differential polynomial of degree $d\leq n-2$, using (\ref{ds27}), we deduce that $\parallel m\left(r, Q_{i,d}(f)/f^{n-1}\right)=O(\log r)$. Consequently $\parallel m(r,\psi_{1i})=O(\log r)$ and so
\bea\label{ds30}\parallel\;T(r, \psi_{1i})=O(\log r).\eea

Also from (\ref{ds28a}), we have
\bea\label{ds31} \partial_{z_i}(f)=\frac{\partial_{z_i}(p_1)+p_1 \alpha_{1i}}{n p_1} f+\frac{\psi_{1i}}{n p_1}.\eea

Now using (\ref{ds31}) and (\ref{dd4}), we get 
\bea\label{ds32}\psi_{1i}f^{n-1}+Q_{i,d}(f)=\beta_{1i}e^{\langle \beta,\ol z\rangle}.\eea

Applying Lemma \ref{L.5} to (\ref{ds32}), we have 
\bea\label{ds33}\beta_{1i}e^{\langle \beta,\ol z\rangle}=\psi_{1i} f^{n-1}+Q_{i,d}(f)=\psi_{1i}\left(f+\beta_{2i}/(n-1)\right)^{n-1},\eea
where $\beta_{2i}$ is a meromorphic function such that $\parallel T\left(r,\beta_{2i}\right)=o(T(r,f))$.
If $\beta_{2i}\equiv 0$, then from (\ref{ds33}), we get $Q_{i,d}(f)\equiv 0$ and so proceeding similarly as done in the proof of Case 2 of Theorem \ref{t1}, we get a contradiction. Hence $\beta_{2i}\not\equiv 0$ and so from (\ref{dd6}) and (\ref{ds33}), we get
\bea\label{ds33a} \ol N\left(r,0,f+\beta_{2i}/(n-1)\right)\leq N\left(r,0,\beta_{1i}\right)=O(\log r).\eea

Now using (\ref{ds32}) to (\ref{dd2}), we get
\bea\label{ds34} p_1e^{\langle \alpha,\ol z\rangle}=f^n+P_d(f)-\frac{p_2 \psi_{1i}}{\beta_{1i}} f^{n-1}-\frac{p_2}{\beta_{1i}} Q_{i,d}(f).\eea

Again using Lemma \ref{L.5} to (\ref{ds34}), we have
\bea\label{ds35} p_1e^{\langle \alpha,\ol z\rangle}=f^n+P_d(f)-\frac{p_2\psi_{1i}}{\beta_{1i}} f^{n-1}-\frac{p_2}{\beta_{1i}} Q_{i,d}(f)=\left(f+\beta_{3i}/n\right)^n,\eea
where $\beta_{3i}$ is a meromorphic function such that $\parallel T(r,\beta_{3i})=o(T(r, f))$.
Clearly from (\ref{ds35}), we have $\beta_{3i}=-\left(p_2 \psi_{1i}\right)/\beta_{1i}$ and
\bea\label{ds33b} \ol N\left(r,0,f+\beta_{3i}/n\right)\leq O(\log r).\eea

If $\beta_{2i}/(n-1)\not\equiv\beta_{3i}/n$, then using (\ref{ds33a}) and (\ref{ds33b}), we conclude from Lemma \ref{L.2} that
$\parallel T(r,f)\leq o(T(r,f))$, which is impossible. Hence $\beta_{2i}/(n-1)\equiv \beta_{3i}/n$. Therefore from (\ref{ds33}) and (\ref{ds35}), we have respectively
\bea\label{ds37}\psi_{1i}(z)\left(f(z)-p_2(z)\psi_{1i}(z)/n\beta_{1i}(z)\right)^{n-1}=\beta_{1i}(z)e^{\langle \beta,\ol z\rangle}\eea
and
\bea\label{ds38}\left(f(z)-p_2(z)\psi_{1i}(z)/n\beta_{1i}(z)\right)^n=p_1(z)e^{\langle \alpha,\ol z\rangle}.\eea

Clearly from (\ref{ds37}) and (\ref{ds38}), one can easily get $\alpha_{1k}/\alpha_{2k}=n/(n-1)$ and
\bea\label{ds39} f(z)=\frac{p_2(z)\psi_{1i}(z)}{n\beta_{1i}(z)}+q(z)e^{\langle \alpha,\ol z\rangle/n},\eea
where $q(z)$ is a polynomial in $\mathbb{C}^m$. Now using (\ref{ds38}) to (\ref{ds31}), we have
\beas\label{ds40} &&\frac{\partial_{z_i}(p_1(z))+p_1(z)\alpha_{1i}}{np_1(z)}q(z)e^{\langle \alpha,\ol z\rangle/n}+\frac{\partial_{z_i}(p_1(z))+p_1(z)\alpha_{1i}}{np_1(z)} \cdot \frac{p_2(z)\psi_{1i}(z)}{n\beta_{1i}(z)}+\frac{\psi_{1i}(z)}{np_1(z)}\nonumber \\ &&=\left(\partial_{z_i}(q(z))+q(z)\alpha_{1i}/n\right)e^{\langle \alpha,\ol z\rangle/n}+\partial_{z_i}\left(p_2(z)\psi_{1i}(z)/n\beta_{1i}(z)\right)\eeas
and so
\bea\label{ds40} \frac{\partial_{z_i}(p_1)+p_1 \alpha_{1i}}{n p_1} \cdot \frac{p_2 \psi_{1i}}{n\beta_{1i}}+\frac{\psi_{1i}}{n p_1}=\partial_{z_i}\left(\frac{p_2 \psi_{1i}}{n \beta_{1i}}\right).\eea

Then using (\ref{dd6}) to (\ref{ds40}), we get
\bea\label{ds41}n\partial_{z_i}\left(p_2 \psi_{1i}/\beta_{1i}\right) =\left((n+1)\partial_{z_i}(p_1)/p_1-n\partial_{z_i}(p_2)/p_2+2 \alpha_{1i}\right)\left(p_2\psi_{1i}\beta_{1i}\right).\eea

Now from (\ref{ds41}), we get
\bea\label{ds42} (p_2^n/p_1^{n+1})\left(p_2 \psi_{1i}/\beta_{1i}\right)^n=e^{2\alpha_{1i}z_i+g_i(z)},\eea
where $g_i(z)=g(z_1,\ldots, z_{i-1},z_{i+1},\ldots,z_m)$ is a polynomial.
Clearly $2\alpha_{1i} z_i+g_i(z)$ is non-constant polynomial and so $e^{2\alpha_{1i}z_k+g_i(z)}$ is a transcendental entire function. Note that $p_1$, $p_2$ and $\beta_{1i}\;(i=1,2,\ldots,m)$ are non-zero polynomials. Therefore using (\ref{ds30}) to (\ref{ds42}), we get  
\[\parallel T(r,e^{2\alpha_{1i}z_i+g_i(z)})=O(\log r),\]
which contradicts the fact that $e^{2\alpha_{1i}z_i+g_i(z)}$ is transcendental.

\smallskip
{\bf Case 2.} Let $\varphi_{ji}\equiv 0$ for all $i,j\in\mathbb{Z}[1,m]$. Now from (\ref{dd10}) and (\ref{dd11}), we get
\beas\label{ds26}\beta_{1i} \partial_{z_i}(Q_{i,d}(f))-\left(\partial_{z_i}(\beta_{1i})+\beta_{1i}\alpha_2\right)Q_{i,d}(f)=0.\eeas

For $Q_{i,d}(f)\not \equiv 0$, we follow the proof of Case 1 of Theorem \ref{t1} to get a contradiction and for $Q_{i,d}(f)\equiv 0$, we follow Case 2 of Theorem \ref{t1} to get a contradiction.
Hence the proof.

\section {{\bf Proof of Theorem \ref{t3}}}
Let $f$ be a transcendental entire solution of (\ref{dt41}). Clearly by Lemma \ref{L.8}, we have $\rho(f)<+\infty$. Applying the operator $\partial_{z_i}$ on the both sides of (\ref{dt41}), we have
\bea\label{dt43}\lambda_i^{-1}\left(3f^2\partial_{z_i}(f)+\partial_{z_i}\left(\partial_u^l(f)\right)\right)=p_1e^{\langle \lambda, \ol z\rangle}-p_2e^{-\langle \lambda, \ol z\rangle}\eea
for all $i\in\mathbb{Z}[1,m]$. Subtracting the square of (\ref{dt43}) from the square of (\ref{dt41}), we get
\bea\label{dt44}f^4a_i=Q_i(f),\eea
where 
\bea\label{dt46} a_i=\lambda_i^2 f^2-9(\partial_{z_i}(f))^2\eea 
and
\[Q_i(f)=4\lambda_i^2 p_1 p_2-2\lambda_i^2f^3\partial_u^l(f)-\lambda_i^2\left(\partial_u^l(f)\right)^2+6f^2\partial_{z_i}(f)\partial_{z_i}\left(\partial_u^l(f)\right)+(\partial_u^l(f))^2,\]
where $i=1,2,\ldots,m$. Now we divide following two cases.\par

\smallskip
{\bf Case 1.} Let $a_i\equiv 0$ for any $i\in\mathbb{Z}[1,m]$. Then from (\ref{dt46}), we get
\bea\label{dt45} \partial_{z_i}(f)\equiv \pm \lambda_i f/3\eea
holds for any $i\in\mathbb{Z}[1,m]$. Then from (\ref{dt45}), we conclude that $f$ and $\partial_{z_i}(f)$ both have no zeros for any $i\in\mathbb{Z}[1,m]$. If we take $f=e^{\xi},$ where $\xi$ is a non-constant entire function in $\mathbb{C}^m$, then from (\ref{dt45}), we get $\partial_{z_i}(\xi)=\pm \lambda_i/3$ for any $i\in\mathbb{Z}[1,m]$.
By Taylor expansion, we have
\bea\label{bm} \xi(z)=\sum\limits_{\mu_1,\ldots,\mu_m=0}^{\infty} b_{\mu_1\ldots\mu_m} z_1^{\mu_1}\ldots z_m^{\mu_m},\eea
where the coefficient $b_{\mu_1\ldots\mu_m}$ is given by
\bea\label{bma} b_{\mu_1\ldots\mu_m}=\frac{1}{\mu_1!\ldots \mu_m!}\frac{\partial^{|I|}\xi(0,0,\ldots,0)}{\partial z_1^{\mu_1}\cdots \partial z_m^{\mu_m}}.\eea

Now from (\ref{bm}) and (\ref{bma}), we have $\xi(z)=A_0\pm\langle \lambda, \ol z\rangle/3$, where $A_0=\alpha(0,0,\ldots,0)$. Therefore
\bea\label{dt45b} f(z_1,\ldots,z_m)=Ae^{\pm\langle \lambda, \ol z\rangle/3},\eea
where $A=\exp(A_0)$. Now using (\ref{dt45b}) (taking positive sign) to (\ref{dt41}), we get
\beas (A^3-p_1)e^{2\langle \lambda, \ol z\rangle}+A\left(\sideset{}{_{k=1}^m}{\sum} u_k\lambda_k/3\right)^le^{4\langle \lambda, \ol z\rangle/3}=p_2,\eeas
which is impossible by Lemma \ref{L.3}. Similarly we get a contradiction if we take negative sign.\par

\smallskip
{\bf Case 2.} Let $a_i\not\equiv 0$ for some $i\in\mathbb{Z}[1,m]$. For the sake of simplicity, we may assume that $a_j\equiv 0$ and $a_k\not\equiv 0$. Then from (\ref{dt46}), we have $f=\pm (\lambda_j/3)\partial_{z_j}(f)$. Taking positive sign, we get $f(z)=e^{(\lambda_j/3)z_j+g_j(z)}$, where $g_j(z)=g(z_1,z_2,\ldots,z_{j-1},z_{j+1},\ldots,z_m)$ is a polynomial. Note that 
\bea\label{ps}f^4a_k=Q_k(f),\eea
where $a_k=\lambda_k^2 f^2-9(\partial_{z_k}(f))^2\not\equiv 0$ and so $Q_k(f)\not\equiv 0$. Substituting $f(z)=e^{(\lambda_j/3)z_j+g_j(z)}$ into (\ref{ps}), we get 
\bea\label{pps} b_6(z)e^{6(\lambda_j/3)z_j+6g_j(z)}+b_5(z)e^{5(\lambda_j/3)z_j+5g_j(z)}+\ldots+b_1(z)e^{(\lambda_j/3)z_j+g_j(z)}+b_0(z)=0,\eea
where $b_i\;(i=0,1,2,\ldots,6)$ are polynomials in $\mathbb{C}^m$ such that $b_6\not\equiv 0$. Now using Lemma \ref{L.2} to (\ref{pps}), we get a contradiction.
Therefore we conclude that $a_i\not\equiv 0$ for all $i\in\mathbb{Z}[1,m]$. Now using Lemma \ref{L.4} to (\ref{dt44}), we get $m(r,a_i)=O(\log r)$ and so $a_i$ is a polynomial for any $i\in\mathbb{Z}[1,m]$. We claim that $\partial_{z_j}(a_i)\equiv 0$ for any $i,j\in\mathbb{Z}[1,m]$. If not, suppose $\partial_{z_j}(a_i)\not\equiv 0$ for any $i,j\in\mathbb{Z}[1,m]$. Applying the operator $\partial_{z_i}$ on the both sides of (\ref{dt46}), we get
\bea\label{bm1} \partial_{z_i}(f)(2\lambda_i^2 f-18\partial^2_{z_i}(f))=\partial_{z_i}(a_i)\not\equiv 0.\eea

Clearly (\ref{bm1}) gives $N(r,0;1/\partial_{z_i}(f))=o(T(r,f))$. Now applying Lemmas \ref{L.2} and \ref{L.3} to (\ref{dt46}), we get
\beas 2T(r,f)&=&T(r,\lambda_i^2 f^2/a_i)+o(T(r,f))\\&\leq& \ol N(r,\lambda_i^2 f^2/a_i)+\ol N(r,0;\lambda_i^2 f^2/a_i)+\ol N(r,1;\lambda_i^2 f^2/a_i)+o(T(r,f))\\&\leq &
N(r,0;f)+N(r,0;1/\partial_{z_i}(f))+o(T(r,f))\leq T(r,f)+o(T(r,f)),
\eeas
which is impossible. Hence $\partial_{z_i}(a_i)\equiv 0$ and so from (\ref{bm1}), we get
\bea\label{bm2} f=(9/\lambda_i^2)\partial^2_{z_i}(f)\eea
holds for any $i\in\mathbb{Z}[1,m]$. We see that $f(z)=ce^{\langle \lambda, \ol z\rangle/3}+de^{-\langle \lambda, \ol z\rangle/3}$
is a solution of (\ref{bm2}), where $c, d\in\mathbb{C}\backslash \{0\}$. Obviously this function $f$ satisfies the following equation
\bea\label{bm3} \lambda_i^2(\partial_{z_j}(f(z)))^2-\lambda_j^2(\partial_{z_i}(f(z)))^2\equiv 0\eea
for all $z\in\mathbb{C}^m$. Now applying the operator $\partial_{z_j}\;(i\neq j)$ on the both sides of (\ref{dt46}), we get
\beas 2\lambda_i^2 f\partial_{z_j}(f)-18\partial_{z_i}(f)\partial_{z_j}(\partial_{z_i}(f))=\partial_{z_j}(a_i)\eeas
and so from (\ref{bm2}), we see that the function $f(z)=ce^{\langle \lambda, \ol z\rangle/3}+de^{-\langle \lambda, \ol z\rangle/3}$  also satisfies the Eq.
\bea\label{bm4} \partial_{z_j}\left(\lambda_i^2(\partial_{z_j}(f))^2-\lambda_j^2(\partial_{z_i}(f))^2\right)=\partial_{z_j}(\lambda_j^2 a_i/9),\eea
for any $i,j\in\mathbb{Z}[1,m]$, $i\neq j$. Therefore from (\ref{bm3}) and (\ref{bm4}), we have $\partial_{z_j}(\lambda_j^2 a_i(z)/9)=0$ for all $z\in\mathbb{C}^m$, i.e., $\partial_{z_j}(a_i)\equiv 0$ for any $i,j\in\mathbb{Z}[1,m]$, $i\neq j$. Finally $\partial_{z_j}(a_i)\equiv 0$ and so $a_i$ is a constant, where $i\in \mathbb{Z}[1,m]$. 
Now from (\ref{dt46}), we see that the functions $\lambda_i f+3\partial_{z_i}(f)$ and $(\lambda_i f-3\partial_{z_i}(f))/a_i$ are of finite order having no zeros. Therefore we may assume that  
\bea\label{xx1} \lambda_i f+3\partial_{z_i}(f)=c_ie^{\xi_i}\;\;\text{and}\;\;\lambda_i f-3\partial_{z_i}(f)=d_ie^{-\xi_i},\eea
where $\xi_i$ is a non-constant polynomial in $\mathbb{C}^m$ and $c_i,d_i\in\mathbb{C}\backslash \{0\}$, where $i\in \mathbb{Z}[1,m]$.
Solving (\ref{xx1}) for $f$ and $\partial_{z_i}(f)$, we get
\bea\label{xx3} f=\left(c_ie^{\xi_i}+d_ie^{-\xi_i}\right)/2\lambda_i\;\;\text{and}\;\;\partial_{z_i}(f)=\left(c_ie^{\xi_i}-d_ie^{-\xi_i}\right)/6,\eea
for any $i\in \mathbb{Z}[1,m]$. Now we claim that $\xi_i\pm \xi_j$ is constant for any $i,j\in\mathbb{Z}[1,m]$, $i\neq j$. If not, suppose $\xi_i\pm \xi_j$ is non-constant for $i,j\in\mathbb{Z}[1,m]$, $i\neq j$. If we take $\hat{c_i}=c_i/2\lambda_i$ and $\hat{d_i}=d_i/2\lambda_i$, then from (\ref{xx3}), we have
\bea\label{xx4} \hat{c}_ie^{2\xi_i+\xi_j}-\hat{c_j}e^{2\xi_j+\xi_i}-\hat{d_j}e^{\xi_i}+\hat{d_i}e^{\xi_j}=0,\eea 
for any $i,j\in\mathbb{Z}[1,m]$, $i\neq j$. Then using Lemma \ref{L.7} to (\ref{xx4}), we get that atleast one of $2\xi_i+\xi_j$ and $2\xi_j+\xi_i$ is a constant. Suppose $2\xi_i+\xi_j=c_{ij}\in\mathbb{C}$ for any $i,j\in\mathbb{Z}[1,m]$, $i\neq j$. Then from (\ref{xx4}) gives 
$\hat c_ie^{c_{ij}}-\hat c_je^{2c_{ij}}e^{-3\xi_i}-\hat d_je^{\xi_i}+\hat d_ie^{c_{ij}}e^{-2\xi_i}=0$, i.e.,
\beas \hat c_ie^{c_{ij}}e^{3\xi_i}-\hat c_je^{2c_{ij}}-\hat d_je^{4\xi_i}+\hat d_ie^{c_{ij}}e^{\xi_i}=0,\eeas 
which is impossible by Lemma \ref{L.3}. Similarly we get a contradiction, when $\xi_i+2\xi_j$ is a constant.
Hence $\xi_i\pm \xi_j$ is constant for any $i,j\in\mathbb{Z}[1,m]$, $i\neq j$. Then from (\ref{xx3}), we may assume that
\bea\label{xx5} f=\left(\tilde c_ie^{\xi}+\tilde d_ie^{-\xi}\right)/2\lambda_i\;\;\text{and}\;\;\partial_{z_i}(f)=\left(\tilde c_ie^{\xi}-\tilde d_ie^{-\xi}\right)/6,\eea
where $\tilde c_i, \tilde d_i\in\mathbb{C}\backslash \{0\}$ for any $i\in\mathbb{Z}[1,m]$. Clearly (\ref{xx5}) gives
\[\partial_{z_i}(f)=\partial_{z_i}(\xi)\left(\tilde c_ie^{\xi}-\tilde d_ie^{-\xi}\right)/2\lambda_i.\]

Now from (\ref{xx5}), we get
$\lambda_i=3\partial_{z_i}(\xi)$, $i\in\mathbb{Z}[1,m]$. Clearly from (\ref{bm}) and (\ref{bma}), we have $\xi(z)=A_0+\langle \lambda, \ol z\rangle/3$, where $A_0=\alpha(0,0,\ldots,0)$. Again from (\ref{xx5}), we have
\bea\label{xx6} (\tilde c_i-\tilde c_j)e^{\xi}-(\tilde d_i-\tilde d_j)e^{-\xi}=0,\eea 
for any $i,j\in\mathbb{Z}[1,m]$, where $i\neq j$. Then using Lemma \ref{L.3} to (\ref{xx6}), we get $\tilde c_i=\tilde c_j$ and $\tilde d_i=\tilde d_j$ for any $i,j\in\mathbb{Z}[1,m]$, $i\neq j$. Finally, we have
\bea\label{xx7} f(z)=ce^{\langle \lambda, \ol z\rangle/3}+de^{-\langle \lambda, \ol z\rangle/3},\eea
where $c, d\in\mathbb{C}\backslash \{0\}$. Substituting (\ref{xx7}) into (\ref{dt41}), we get
\beas\left(c^3-p_1\right)e^{\langle \lambda, \ol z\rangle}+\left(d^3-p_2\right)e^{-\langle \lambda, \ol z\rangle}+(3cd+m^l)\left(ce^{\langle \lambda, \ol z\rangle/3}+de^{-\langle \lambda, \ol z\rangle/3}\right)=0\eeas
and so by Lemma \ref{L.3}, we have $c^3=p_1$, $d^3=p_2$ and $3cd+m^l=0$.
Hence the proof.

\section {{\bf Proof of Theorem \ref{t4}}}
If $l$ is odd, then substituting (\ref{xx7}) into (\ref{dt41}), one can easily get a contradiction. Consequently the Eq.
(\ref{de}) has no transcendental entire solutions in $\mathbb{C}^m$.

\medskip
\medskip
{\bf Statements and declarations:}

\smallskip
\noindent \textbf {Conflict of interest:} The authors declare that there are no conflicts of interest regarding the publication of this paper.

\smallskip
\noindent{\bf Funding:} There is no funding received from any organizations for this research work.

\smallskip
\noindent \textbf {Data availability statement:}  Data sharing is not applicable to this article as no database were generated or analyzed during the current study.

\end{document}